\documentclass[iccmp]{ipbook}

\makeatletter
\AtBeginDocument{\let\publname\@empty\let\@serieslogo\@empty}
\makeatother

\startlocaldefs

\usepackage{amssymb,amsfonts}

\usepackage{mathrsfs}
\usepackage{bbm}
\usepackage{hyperref}
\hypersetup{
    colorlinks,
    linktocpage,
    linkcolor={blue},
    citecolor={blue}
}

\numberwithin{equation}{section}

\newtheorem{theoremcounter}{theoremcounter}[section]

\theoremstyle{plain}
\newtheorem{theorem}[theoremcounter]{Theorem}
\newtheorem{lemma}[theoremcounter]{Lemma}
\newtheorem{proposition}[theoremcounter]{Proposition}
\newtheorem{corollary}[theoremcounter]{Corollary}
\newtheorem{example}[theoremcounter]{Example}

\theoremstyle{definition}
\newtheorem{definition}[theoremcounter]{Definition}

\theoremstyle{remark}
\newtheorem{remark}[theoremcounter]{Remark}

\newcommand{\Cs}{\mathrm{C}^*}
\newcommand{\Csr}{\mathrm{C}^*_{\mathrm{red}}}

\newcommand{\bC}{\mathbb C}

\newcommand{\bZ}{\mathbb Z}

\newcommand{\ra}{\rightarrow}

\endlocaldefs

\firstpage{1}
\lastpage{1}

\title[]{Higher Kazhdan projections for one-relator groups}
\author[]{Sanaz Pooya}
\address{Institute of Mathematics, 
		University of Potsdam, 
		14476 Potsdam, Germany}
\email{sanaz.pooya@uni-potsdam.de}\thanks{}
\author[]{Baiying Ren}
\address{Research Center of Operator Algebras, East China Normal University, Shanghai 200241, China; Currently: Mathematisches Institut, Universit{\"a}t G{\"o}ttingen, G{\"o}ttingen 37073, Germany}
\email{52275500020@stu.ecnu.edu.cn}
\author[]{Hang Wang}
\address{Research Center of Operator Algebras, East China Normal University, Shanghai 200241, China}
\email{wanghang@math.ecnu.edu.cn}

\begin{document}

\begin{abstract}
We study the $K$-theory classes of higher Kazhdan projections of a discrete group under a spectral-gap assumption and recall how they are obtained, via the Baum--Connes assembly map, from equivariant Euler classes.  After reviewing the case of virtually free groups, where the universal space for proper actions has a one-dimensional model, we prove the main new result: an explicit description of the $K$-theory class of the first higher Kazhdan projection for finitely generated infinite one-relator groups whose first Laplacian has a spectral gap.  If moreover the one-relator group is hyperbolic, this also gives corresponding formulas for delocalised $\ell^2$-Betti numbers.
\end{abstract}

\maketitle


	
\section{Introduction}
\label{sec:introduction}

The Euler characteristic is a fundamental invariant in topology.
For a closed manifold $M$, it is defined homologically as the alternating sum of Betti numbers,
\[
\chi(M)=\sum_i (-1)^i \dim H_i(M).
\]
If $M$ admits a finite triangulation, the same invariant is computed combinatorially as the alternating sum of the numbers of simplices in each dimension. Hodge theory provides a third, analytic description. For a closed Riemannian manifold $M$, the de Rham operator
\[
d+d^*:\Omega^{\mathrm{even}}(M)\longrightarrow
\Omega^{\mathrm{odd}}(M)
\]
is Fredholm, and its index equals $\chi(M)$. Equivalently, $\chi(M)$ is the alternating sum of the dimensions of spaces of harmonic forms, arising as the kernels of the Laplacians on differential forms.

In the 1970s, Atiyah formulated the theory of $K$-homology, the dual of $K$-theory, using equivalence classes of elliptic operators~\cite{Atiyah1}. Noting that de Rham operators form a distinguished class of elliptic operators, one may study the Euler characteristic using its $K$-homology class, which provides greater flexibility in working with operators. See \cite{Rosenberg99}. This point of view naturally extends to the equivariant setting and lands in the framework of operator
algebras \cite{Lueck-Rosenberg,EM}. Given a discrete group $G$ of type $F_n$ for some $n$, 
and a unitary representation $\mathcal H$ of $G$, the role of harmonic forms is played by
the kernels of higher degree Laplacians associated with free resolutions of the
trivial $\bZ G$-module $\mathbb Z$, with coefficients in $\mathcal H$ \cite{linowakpooya2020}. More precisely, let $\Delta_n$ denote the
Laplacian in degree $n$ defined by the differential of the cochain complex associated with the free resolution. When $\Delta_n$ has a spectral gap at zero, the
orthogonal projection onto its kernel defines the higher Kazhdan projection
$p_n$.  
The spectral gap assumption
is essential, since it guarantees that the projection obtained as the limit
of heat operators
\[
p_n=\lim_{t\to\infty}e^{-t\Delta_n}
\]
belongs to the relevant group $C^*$-algebra.  
Criteria for the existence of
these spectral gaps in terms of reduced group cohomology are given in
\cite{Bader-Nowak}, and their relation with Novikov--Shubin invariants is
described in \cite{LueckL2inv}.
Higher Kazhdan projections were also introduced and studied in connection with $\ell^2$-Betti numbers and the (coarse)
Baum--Connes conjecture \cite{linowakpooya2020}. 

In this paper, we focus on the left regular representation and hence on
higher Kazhdan projections in $K$-theory 
\[
[p_n]\in K_0(C^*_r(G)).
\]
These classes provide a noncommutative analogue of the contribution of
harmonic forms to the Euler characteristic.  Pairing them with traces gives
numerical invariants.  In particular, using delocalised traces
$\tau_{\langle g\rangle}$ on suitable smooth subalgebras of $C^*_r(G)$
(see \cite{puschnigg2010}), one obtains the delocalised
$\ell^2$-Betti numbers
\[
\beta^{(2)}_{n,\langle g\rangle}(G)
=
\tau_{\langle g\rangle}([p_n]),
\]
introduced in this setting in \cite{Pooya-Wang}.  Thus, higher Kazhdan
projections provide an operator-algebraic realisation of $L^2$-cohomological
information.

The purpose of this paper is to understand the $K$-theory classes of higher
Kazhdan projections more explicitly.  Although these projections are defined
analytically as spectral projections of higher Laplacians, their classes in
$K_0(C^*_r(G))$ are generally difficult to compute directly \cite{Pooya-Wang}.  Our approach is
to identify these analytic classes through equivariant Euler
characteristics.  More precisely, we relate the combinatorial equivariant
Euler class introduced by Emerson and Meyer \cite{EM} to the higher Kazhdan
projection classes under the Baum--Connes assembly map.
See Theorem \ref{main thm1} (also the main result in \cite{Pooya-Ren-Wang}).  Let $G$ be a discrete group
admitting a $G$-finite model for the universal proper $G$-space
$\underline EG$.  Assume that there is only one non-vanishing higher Kazhdan
projection $p_n$.  Then the Baum--Connes assembly map $\mu^G$ sends the combinatorial
equivariant Euler class $[\mathrm{Eul}^{\mathrm{cmb}}]$ to the class of the higher Kazhdan projection:
\[
\mu^G([\mathrm{Eul}^{\mathrm{cmb}}])
=
(-1)^n[p_n].
\]
Moreover, the class of $p_n$ admits the explicit description
\[
(-1)^n[p_n]
=
\sum_{\sigma\in G\backslash SX}
(-1)^{|\sigma|}[\rho_\sigma],
\]
where $\rho_\sigma$ is the averaging projection associated with the stabiliser
subgroup of the simplex $\sigma$.  More generally, when several higher
Kazhdan projections are non-zero, the same argument identifies the Euler class
with the alternating sum of the $K$-theory classes of higher Kazhdan projections.

We first apply this result to non-amenable finitely generated virtually free
groups \cite{Pooya-Ren-Wang}.  Such groups admit Bass--Serre trees as cocompact models for
$\underline EG$.  In this case the only non-vanishing higher Kazhdan
projection is the one in degree one; see Lemma \ref{lem:existence}.  Theorem
\ref{main thm1} therefore gives a concrete formula for its $K$-theory class in
terms of the averaging projections associated with the vertex and edge
stabilisers of the Bass--Serre tree, as stated in Corollary
\ref{cor: tree}.  This formula leads to explicit computations for amalgamated
free products and HNN extensions, see Corollaries \ref{cor: amalgam} and
\ref{cor: HNN}.  In particular, for
\[
\mathrm{SL}(2,\mathbb Z)=\mathbb Z_4*_{\mathbb Z_2}\mathbb Z_6,
\]
we obtain an explicit expression for $[p_1]$ in terms of averaging projections
of finite subgroups (Example \ref{ex: sl2r}).  Applying delocalised traces then
gives non-vanishing results for delocalised $\ell^2$-Betti numbers, including
the examples described in Corollary \ref{cor: nonzero del betti}.

We then study finitely generated infinite one-relator groups
\[
G=\langle X\mid r\rangle,
\]
where $X$ is a finite set of generators and $r$ is a single relator. This class includes important examples such as fundamental groups of closed surfaces.
The $\ell^2$-Betti numbers of such groups are concentrated in degree one by
the result of Dicks and Linnell \cite{Warren-Linnell}.  However, unlike the
virtually free case, the existence of the first higher Kazhdan projection
depends on the spectral gap of the first Laplacian, which is a more subtle
condition.  Assuming that this spectral gap exists, Theorem
\ref{theorem: one relator groups} gives the formula
\[
[p_1]=( |X|-1 )[1]-[e]\in K_0(C^*_r(G)),
\]
where $e$ is the averaging projection associated with the finite cyclic
subgroup determined by the proper-power decomposition of the relator $r$; see
also \cite{Chiswell04}.  When the group is hyperbolic, this description
immediately gives explicit formulas for delocalised $\ell^2$-Betti numbers
through Corollary \ref{corollary: l2 Betti numbers for one relator groups}.
We illustrate this result with closed hyperbolic surface groups, the groups
$\mathbb Z_n*\mathbb Z$, and further examples where the spectral gap
condition fails, including $\mathbb Z^2*\mathbb Z$.  We also discuss
Baumslag--Solitar groups, where all higher Kazhdan projections vanish.

The paper is organised as follows.  In Section \ref{sec: Kazhdan
projections} we recall higher Kazhdan projections, spectral gap criteria, and
delocalised $\ell^2$-Betti numbers.  Section \ref{sec: Main theorem} discusses
the relation between the combinatorial Euler characteristic and higher
Kazhdan projection under the Baum--Connes assembly map.
Section~\ref{sec:4} and Section~\ref{sec:5} focus on the applications to virtually free groups and one-relator groups, respectively.

    \section*{Acknowledgements}
HW is supported
by the grants 23JC1401900 and NSFC 12271165, and in part by the Science and Technology Commission of
Shanghai Municipality (No. 22DZ2229014).
BR is supported by NSFC grant 125B2009 and a scholarship from the China Scholarship Council (No. 202506140038).
The authors would like to thank Thomas Schick for very helpful suggestions on the Novikov--Shubin invariants of one-relator groups.

    \section{Higher Kazhdan projections and delocalised \texorpdfstring{$\ell^2$}{l2}-Betti numbers}
\label{sec: Kazhdan projections} 

In this section, we recall higher Kazhdan projections in a generalized setting as introduced in \cite{Pooya-Ren-Wang}, and the spectral-gap condition ensuring that they belong to the relevant group $C^*$-algebra. We review criteria for spectral gaps in terms of reduced cohomology and Novikov--Shubin invariants, and then recall delocalised traces and the associated delocalised $\ell^2$-Betti numbers.

   \begin{definition}\cite[Definition 2.1]{Pooya-Ren-Wang} \label{def: higher kazhdan gen}
    Let $G$ be a group of type $FP_{n+1}$ over $\mathbb C$. Let $\mathcal F $ be a family of unitary representations of $G$. A higher Kazhdan projection $p_n$ in degree $n$ is a projection in (a corner of) 
    $\mathrm{M}_{k_n}(\Cs_ {\mathcal F}(G))$ for some $k_n \in \mathbb N$ whose image under any unitary representation $(\pi, H) \in \mathcal F$ is the orthogonal projection 
    \[\pi(p_n) \colon C^n(G, H) \ra \ker \pi(\Delta_n).
    \]
\end{definition}

For each $(\pi,H)\in\mathcal F$, the higher Kazhdan projection satisfies
$\pi(p_n)=\lim_{t\to\infty}e^{-t\pi(\Delta_n)}$ in the strong operator topology. If the family $\{\pi(\Delta_n):(\pi,H)\in\mathcal F\}$ has a uniform spectral gap at zero, these heat operators converge in the norm defining $\Cs_{\mathcal F}(G)$ and yield
$p_n\in M_{k_n}(\Cs_{\mathcal F}(G))$.
A characterisation of the existence of a spectral gap for the higher Laplace operator $\Delta_n$ in terms of reduced group cohomology is given in \cite{Bader-Nowak}. This will be our first tool for verifying the existence of spectral gaps.
\begin{proposition} \label{prop: spectral gap tool}
\cite[Proposition 16]{Bader-Nowak}
    Let $(\pi, H)$ be a unitary representation of $G$. The higher Laplace operator $\pi(\Delta_n)$ has a spectral gap in $\mathrm{M}_{k_n}(\Cs_{\pi}(G))$ if and only if the group cohomology $\mathrm H^n(G, H)$ and $\mathrm H^{n+1}(G, H)$ are both reduced. 
\end{proposition}

In this article, we work with the left regular representation $(\lambda, \ell^2(G))$.
In this case, the reducedness of the first group cohomology is characterised as follows.

\begin{lemma}[{\normalfont \cite[Corollary III.2.4]{Guichardet}; see also page 257 of \cite{LueckL2inv}}]
\label{lemma: the first group cohomology}
    Let $G$ be a discrete group.
    Then $\mathrm H^1(G, \ell^2(G))$ is reduced if and only if $G$ is non-amenable. 
\end{lemma}
Another criterion for the existence of spectral gaps is provided by the Novikov--Shubin invariants.
See \cite{LueckL2inv} for a comprehensive introduction.
Novikov--Shubin invariants $\alpha_p(M)$ measure the behaviour of the spectrum of the analytic Laplace operators $\Delta_p$ around zero for a cocompact, free, proper $G$-manifold $M$ with a $G$-invariant Riemannian metric.
Furthermore, as introduced in \cite{LueckL2inv}, $\alpha_p(M)$ also has an equivalent combinatorial definition when $M$ admits a free $G$-CW-complex model of finite type.
For a discrete group $G$ that is finitely presented and of type $FP_{\infty}$, Theorem 11.38 (ii) of \cite{Lueck2025} provides a CW-complex model $X$ for $EG$ of finite type.
In this setting, the Novikov--Shubin invariants for $X$ precisely characterise the spectrum around zero of the Laplace operators with coefficients in $\ell^2(G)$, as needed.


\begin{proposition}\cite[Lemma 2.66 (2) and Theorem 3.136 (3)]{LueckL2inv}
\label{prop: spectral gap and NS inv}
    Let $G$ be a discrete group and $X$ be a connected free $G$-CW-complex of finite type.
    Then the Laplace operator $\Delta_p$ has a spectral gap at 0 if and only if $\alpha_p(X)=\alpha_{p+1}(X)={\infty}^+$; equivalently, the unreduced and reduced $L^2$-de Rham cohomology agree in dimensions $p$ and $p+1$.
\end{proposition}

\begin{lemma}\cite[Theorem 2.55 (5)]{LueckL2inv}
\label{lemma: NS invariant}
    Let $G$ be a discrete group and $X$ be a connected free $G$-CW-complex of finite type.
    Then $\alpha_1(X)={\infty}^+$ if and only if $G$ is finite or non-amenable. 
\end{lemma}


Higher Kazhdan projections are closely related to delocalised $\ell^2$-Betti numbers, which are obtained by pairing these projections with delocalised traces based on the work by the first author and the third author \cite{Pooya-Wang}.
The delocalised trace $\tau_{\langle{g}\rangle}$ defined on $\ell^1(G)$ is the bounded linear map with the tracial property
\[
\tau_{\langle{g}\rangle} \colon \ell^1(G) \ra \mathbb C, \qquad \tau_{\langle{g}\rangle}(f) = \sum _{h\in \langle{g}\rangle} f(h),
\]
where $\langle{g}\rangle$ denotes the conjugacy class of $g$.

\begin{definition} \cite[Definition 2.11]{Pooya-Wang}
Let $G$ be a discrete group of type $FP_{n+1}$. Assume that there is a smooth subalgebra $\mathcal S \subseteq \Csr(G)$ to which the delocalised trace $\tau_{\langle g\rangle}$ extends. Assume further that the K-theory class $[p_n]$ lies in $\mathrm K_0(\Csr(G))$. The $n$-th delocalised $\ell^2$-Betti number of $G$ is
     \[
        \beta _{n, \langle{g}\rangle}^{(2)} (G) = \tau_{\langle{g}\rangle} ([p_n]).
     \]
     \end{definition}
For example, Puschnigg \cite{puschnigg2010} showed that hyperbolic groups have a smooth subalgebra of $\Csr(G)$ to which delocalised traces extend continuously.

Furthermore, from the operator-algebraic perspective, one may also define $\ell^2$-Betti numbers through the K-theory pairing
\[\beta_n^{(2)} (G) = \tau([p_n]),\] 
where $\tau$ is the von Neumann trace. 
Here $\tau$ picks out the coefficient of the group identity in an $\ell^1(G)$-function. Delocalised $\ell^2$-Betti numbers are therefore in complete analogy with $\ell^2$-Betti numbers.

	\section{Combinatorial Euler characteristics as the preimage}
    \label{sec: Main theorem}

The results in this section are recalled from \cite[Sections 3--5]{Pooya-Ren-Wang}. We retain references to the original sources where appropriate.
In this work, we adopt the definition of the combinatorial Euler characteristic given by Emerson and Meyer \cite{EM}.
We omit the detailed construction of their $KK$-class here. See \cite{EM} for details.
\begin{lemma} [{\normalfont \cite[Lemma 3.3]{Pooya-Ren-Wang}; see also \cite[Section 5]{EM}}]
\label{lemma: Euler}
Let $G$ be a discrete group. Suppose  there exists a cocompact simplicial complex model $X$ for $\underline EG$. Then the image of the combinatorial Euler characteristic $\mathrm{Eul}^{\mathrm{cmb}}_X$ under the Baum–Connes assembly map $\mu^G$ is given by 
 $$
  \mu^G([\mathrm{Eul}^{\mathrm{cmb}}_X])=\sum_{\sigma\in G \backslash SX}{{(-1)}^{|\sigma|}[\rho_{\sigma}]}\in \mathrm{K}_0(\Csr(G)),
 $$
 where $SX$ denotes the set of simplices of $X$,  $|\sigma|$ is the dimension of the simplex $\sigma$, and $\rho_{\sigma}$ is the projection associated with $\sigma$.
\end{lemma}

The following theorem presents our main result, which gives a concrete description of the K-theory class of the higher Kazhdan projection.
The main idea of the proof is to identify, using several technical results, the $K$-class of the specific alternating sum of averaging projections associated with the image of the combinatorial Euler characteristic in Lemma \ref{lemma: Euler} with the alternating sum of higher Kazhdan projections. See Section 4 of \cite{Pooya-Ren-Wang} for details. 



  


\begin{theorem} \cite[Theorem 4.1]{Pooya-Ren-Wang}
\label{main thm1}
    Let $G$ be a discrete group for which there is a $G$-finite model $X$ for $\underline E G$. 
    Assume that there is only one non-vanishing higher Kazhdan projection $p_n$ associated with $G$, for some $n$. (Hence, the spectral gap for the Laplace operator in degree $n$ exists.)
    Then the $\mathrm{K}$-theory class $[p_n]$ is given explicitly in terms of the averaging projections associated with the simplex orbits:
    \begin{equation}
        \label{p1}
    (-1)^n[p_n] = \sum_{\sigma\in G \backslash SX}(-1)^{|\sigma|} [\rho_{\sigma}] \in \mathrm{K_0}(\Csr(G)).
    \end{equation}
     Moreover, the combinatorial Euler class   
   $
    [\mathrm{Eul}^{\mathrm{cmb}}_X] \in 
    \mathrm{K}^G_0(\underline{E}G)
    $
    satisfies
    \begin{equation}
        \label{p2}
        \mu^G([\mathrm{Eul}^{\mathrm{cmb}}_X]) = (-1)^n[p_n].
    \end{equation}
\end{theorem}

\begin{remark} \cite[Remark 4.8]{Pooya-Ren-Wang}
When a group \( G \) of type $FP_{n+1}$ admits more than one non-zero higher Kazhdan projection, then 
the combinatorial Euler class $[\mathrm{Eul}^{\mathrm{cmb}}_X]$ associated with $X$ is the preimage of the alternating sum of higher Kazhdan projections for $G$ under the assembly map $\mu_G$, i.e., 
    \begin{equation*}
        \mu^G([\mathrm{Eul}^{\mathrm{cmb}}_X])
        =
        \sum_{i=1}^{n}(-1)^i[p_i].
    \end{equation*}
Indeed, the proof follows the same steps as that of Theorem \ref{main thm1}. As our current interest is to describe explicitly the K-theory classes of higher Kazhdan projections and identify their preimages under the Baum--Connes assembly map, we state our main theorem under the assumption that there is only one non-zero higher Kazhdan projection. 
\end{remark} 

\section{Applications to virtually free groups}
\label{sec:4}
The results in this section are recalled from \cite[Section 6]{Pooya-Ren-Wang}.
In this section, we apply Theorem \ref{main thm1} to the class of non-amenable finitely generated virtually free groups. These groups have exactly one non-zero higher Kazhdan projection in degree one, as shown in the following lemma \ref{lem:existence}.
Moreover, as we will see, the standard tree on which they act provides the required $G$-simplicial complex for $\underline E G$, so we obtain a concrete description of the $K$-theory class of their higher Kazhdan projection and hence computation of delocalised $\ell^2$-Betti numbers.



    \begin{lemma} [{\normalfont \cite[Lemma 4.1]{Pooya-Wang}; see also \cite[Lemma 6.2]{Pooya-Ren-Wang} }]
     \label{lem:existence}
        Let $G$ be a non-amenable virtually free group. Then $p_1$ lies in a matrix algebra over $\Csr(G)$, and its $K$-class is a non-zero element of
        $\mathrm K_0(\Csr(G))$. All other projections $p_n$ vanish when $n\neq 1$.
	\end{lemma}

\begin{corollary} \cite[Corollary 6.3]{Pooya-Ren-Wang}
\label{cor: tree}
    Let $G$ be a non-amenable finitely generated virtually free group, and
    let $X$ be its Bass--Serre tree. Then the $\mathrm{K}$-theory class $[p_1]$ can be described in terms of the projections associated with the vertex and edge groups of the fundamental domain of the tree $X$:
    \begin{equation*}
       -[p_1]
       =\sum_{v\in G\backslash\operatorname{Vert}(X)}[\rho_v]
        -\sum_{e\in G\backslash\operatorname{Edge}(X)}[\rho_e]
        \in \mathrm K_0(\Csr(G)).
    \end{equation*}
\end{corollary}

Let $\mathcal F_G$ be the additive subgroup of $\mathbb Q$ generated by the inverses of the orders of finite subgroups of $G$, namely,
\[
\mathcal F_G:=\left\langle\left\{\frac{1}{|F|}\;\middle|\;F\leq G\text{ finite}\right\}\right\rangle\leq\mathbb Q.
\]

\begin{corollary} \cite[Corollary 6.8]{Pooya-Ren-Wang} \label{cor: nonzero del betti}
    Let $G$ be a non-amenable finitely generated virtually free group acting properly on a tree $X$, and let 
    $g\in F$ be an element in a finite subgroup $F \leq G $ that fixes vertices, but no edges of $X$.
    Then the first delocalised $\ell^2$-Betti number of $G$ is non-zero and satisfies
    \[
       \beta^{(2)}_{1, \langle{g}\rangle} (G) \in 
        \mathcal {F}_G \subseteq \mathbb{Q}.
       \]
\end{corollary}

We end this section with the following concrete examples of virtually free groups and provide the corresponding computations.

\begin{corollary} \cite[Corollary 6.5]{Pooya-Ren-Wang} \label{cor: amalgam}
    Let $G=F\ast_K L$ be an amalgamated free product of finite groups $F$ and $L$ over a common proper subgroup $K$ (where $K\subsetneqq F$ and $K\subsetneqq L$), with indices $[F:K]\geq 3$ or $[L:K]\geq 3$.
    Then the $K$-theory class $[p_1] \in \mathrm K_0(\Csr(G))$ is
    $$
    [p_1]=
    -[\rho_F]-[\rho_L]+[\rho_K],
    $$
    where $\rho_F$, $\rho_L$, and $\rho_K$ are the averaging projections associated with $F$, $L$, and $K$, respectively.
\end{corollary}

\begin{corollary} \cite[Corollary 6.6]{Pooya-Ren-Wang}
\label{cor: HNN}
    Let $F$ be a finite group with two proper subgroups $A,B\subsetneqq F$, and let $\psi\colon A\to B$ be an isomorphism.
    Let $G$ be the associated HNN extension with respect to $\psi$.
    Then the $K$-theory class $[p_1] \in \mathrm K_0(\Csr(G))$ is
    $$
    [p_1]=
    -[\rho_F]+[\rho_A]=-[\rho_F]+[\rho_B].
    $$
\end{corollary}


\begin{example} \cite[Example 6.9]{Pooya-Ren-Wang}
\label{ex: sl2r}
Consider the group $G = \mathrm{SL}(2, \mathbb Z) = 
    \mathbb Z_4 *_ {\mathbb Z_2} \mathbb Z_6$, with $u$, $s$, and $t$ as generators for $\mathbb Z_2$, $\mathbb Z_4$, and $\mathbb Z_6$, respectively. Corollary \ref{cor: amalgam} implies
$$ [p_1] = \left[\frac{1+u}{2}\right]- \left[\frac{1+s+s^2+s^3}{4}\right] - \left[\frac{1+t+t^2+t^3+ t^4 +t^5}{6}\right].
$$
Then the delocalised $\ell^2$-Betti numbers for $G$ are
		\begin{equation*}
			\beta^{(2)}_{1, \langle{g}\rangle}(G) = 
			\begin{cases}
                1/12 & \text{if }g=1_G,\\
                1/12 & \text{if }g=u,\\
                -1/4 & \text{if }g\text{ is conjugate to }s\text{ or }s^3,\\
                -1/6 & \text{if }g\text{ is conjugate to }t,t^2,t^4,\text{ or }t^5,\\
                0 & \text{otherwise}.
			\end{cases}
		\end{equation*}
		and $\beta ^{(2)}_{k, \langle{g}\rangle}(G) = 0$ for $k \neq 1$ and $g\in G$.
\end{example}

\begin{example}\cite[Example 6.10]{Pooya-Ren-Wang}
    Consider the Klein four group $V=\langle a,b\,|\,a^2=b^2=e,ab=ba \rangle$
     with its two isomorphic proper subgroups $A=\{e,a\}$ and $B=\{e,b\}$. 
    Let $G$ be
     the associated HNN extension.
    Then by Corollary~\ref{cor: HNN}, $[p_1]$ is 
    $$
    [p_1]=\left[\frac{e+a}{2}\right]-\left[\frac{e+a+b+ab}{4}\right].
    $$
    Hence, the delocalised $\ell^2$-Betti numbers are
    \begin{equation*}
        \beta^{(2)}_{1, \langle{g}\rangle}(G) = 
        \begin{cases}
            1/4 &\qquad g=e\\
            -1/4 &\qquad g=ab \\
            0 &\qquad \text{otherwise}
        \end{cases}
    \end{equation*}
    and $\beta ^{(2)}_{k, \langle{g}\rangle}(G) = 0$ for $k \neq 1$ and $g\in G $.
\end{example}

\begin{example}\cite[Example 6.11]{Pooya-Ren-Wang}
    Consider the dihedral group $D_4=\langle r,s \,|\,r^4=s^2=e,srs=r^{-1}\rangle
    $
    with its two isomorphic proper subgroups $A=\{e,s \}$ and $B=\{e,sr \}$.
    Let $G
    $ be the associated HNN extension.
   Then by Corollary~\ref{cor: HNN}, the $K$-theory class $[p_1]$ 
   is
    $$
      [p_1]=\left[\frac{e+s}{2}\right]-\left[\frac{e+r+r^2+r^3+s+sr+sr^2+sr^{3}}{8}\right].
    $$
    Hence, the delocalised $\ell^2$-Betti numbers are
    \begin{equation*}
        \beta^{(2)}_{1, \langle{g}\rangle}(G)=
        \begin{cases}
            3/8 &\qquad g=e \\
            -1/4 &\qquad g\in \langle {r} \rangle=\langle {r^3} \rangle \\
            -1/8 &\qquad g\in \langle {r^2} \rangle \\
            1/4 & \qquad g\in \langle {s}\rangle \\
            -1/4 & \qquad g\in \langle {sr} \rangle \\
            0 &\qquad \text{otherwise}
        \end{cases}
    \end{equation*}
    and $\beta ^{(2)}_{k, \langle{g}\rangle}(G) = 0$ for $k \neq 1$ and $g\in G$.
\end{example}

\section{Applications to one-relator groups}\label{sec:5}
In this section, we apply Theorem \ref{main thm1} to a new class of examples given by finitely generated infinite one-relator groups. Unlike the virtually free groups considered above, these groups may have cohomological dimension two, extending the application of the theorem beyond the one-dimensional Bass--Serre setting.

Suppose that $G$ is a one-relator group with presentation
$$
G=\langle X\mid r\rangle,
$$
where $X$ is a set of generators with $|X|:=d$, and $r$ is an element of the free group $F$ on $X$. 
For the subsequent analysis, we restrict ourselves to the case of finitely generated infinite groups. Hence $1<d<\infty$.
This class includes important examples, such as the fundamental group of a closed orientable surface $\Sigma$ of genus $g$:
$$
\pi_1(\Sigma)=\langle x_1,x_2,\ldots,x_{2g-1},x_{2g}\mid [x_1,x_2]\cdots [x_{2g-1},x_{2g}]\rangle,
$$
where $[x_{2i-1},x_{2i}]$ denotes the commutator $x_{2i-1}x_{2i}{x_{2i-1}}^{-1}{x_{2i}}^{-1}$.

The class of one-relator groups has been studied extensively.
Recall that the relation $r$ is an element of the free group $F$.
Denote by $m:=\exp_F{(r)}$ the exponent of $r$ in $F$. It is defined as the supremum of the integers $k$ such that $r=a^k$ for some element $a\in F$, with $\infty$ allowed. Hence, $m\in\mathbb N\cup\{\infty\}$.
According to Theorem 6.22 and Corollary 6.15 of \cite{Chiswell04}, the Euler characteristic of $G$ with $1<d<\infty$ is given by
$$
\chi(G)=1-d+\frac{1}{m}.
$$
Therefore, $\chi(G)\leq 0$, and it vanishes if and only if $d=2$ and $m=1$.
Moreover, the $\ell^2$-Betti numbers for one-relator groups turn out to be concentrated in degree 1 as shown below.

\begin{lemma}[{\normalfont \cite{Warren-Linnell}}]
\label{lemma: l2 Betti number}
    Let $G$ be a one-relator group as above. Then the $\ell^2$-Betti numbers of $G$ are given by
    \begin{equation*}
        {\beta}^{(2)}_n(G)=
        \begin{cases}
            0 &\qquad \text{if } n=0,\\
            -\chi(G)=d-1-\frac{1}{m} &\qquad \text{if } n=1,\\
            0 &\qquad \text{if } n\geq 2.
        \end{cases}
    \end{equation*}
\end{lemma}

Given the faithfulness of the von Neumann dimension, the kernels of the Laplace operators all vanish except in degree 1. 
Therefore, we only need to consider the higher Kazhdan projections in degree 1 for one-relator groups.
However, despite the simplicity of the $\ell^2$-Betti numbers, the existence of the first higher Kazhdan projection for a one-relator group is a more subtle question, since no general results are known concerning the existence of a spectral gap at 0 for the associated first Laplace operator.

Once we know that the first spectral gap exists for a one-relator group, we can apply our main result, Theorem \ref{main thm1}, to obtain an expression for the $K$-theory class of its first higher Kazhdan projection, since the higher Kazhdan projections in all other degrees vanish.
In this case, we use the projective resolution in Lemma \ref{lemma: proj resolution} directly, rather than a concrete finite CW-complex model for $\underline{E}G$.
Recall that $m:=\exp_F{(r)}$. 
When $m<\infty$, one obtains $r=q^m$ for some $q\in F$ by definition.
Let $c$ denote the image of $q$ in $G$, and $C$ denote the finite cyclic group generated by $c$.
Let $e$ be the averaging projection of $C$.
When $m=\infty$, we set $e=0$.
Hence, $e$ is always a projection in $\bC G$.

\begin{lemma}\cite[Lemma 6.21 and proof of Theorem 6.22]{Chiswell04}
\label{lemma: proj resolution}
    Let $P=\bZ[G/C]$ when $m<\infty$, and let $P=0$ when $m=\infty$. Then the following is an exact sequence of left $\bZ G$-modules:
    \begin{equation*}
        0 \to P \stackrel{\delta_1}{\longrightarrow} \bigoplus_{X}\bZ G \stackrel{\delta_0}{\longrightarrow}\bZ G \stackrel{\epsilon}{\longrightarrow} \bZ \longrightarrow 0.
    \end{equation*}
    This yields a projective resolution of the trivial $\bC G$-module $\bC$:
    \begin{equation}
    \label{proj resolution}
		0 \to \bC Ge \stackrel{\delta_1}{\longrightarrow}  \bigoplus_{X}\bC G \stackrel{\delta_0}{\longrightarrow}\bC G \stackrel{\epsilon}{\longrightarrow} \bC \longrightarrow 0.
	\end{equation}
    Here, for each $x\in X$, $\delta_0(x)$ is the image of $x-1$ in $\bZ G$, and $\delta_1$ is the left Fox derivative.
\end{lemma}


\begin{theorem}
\label{theorem: one relator groups}
    Let $G=\langle X\mid r\rangle$ be a one-relator group.
    Assume that the first spectral gap for $G$ exists.
    Then the $K$-theory class of the first higher Kazhdan projection for $G$ is given by
    $$
    [p_1]=(d-1)[1]-[e]\in \mathrm K_0(\Csr(G)),
    $$
    and $p_1$ is the only non-vanishing higher Kazhdan projection.
    Here $d=|X|\in\mathbb N$ with $d>1$ is the number of generators, and $e$ is the averaging projection associated with the relation $r$.
\end{theorem}
\begin{proof}
    Applying the functor $\mathrm{Hom}(\cdot, \ell^2(G))$ to the (unaugmented) projective resolution (\ref{proj resolution}) in Lemma \ref{lemma: proj resolution}, we obtain a finite-length cochain complex of Hilbert spaces:
    \begin{equation*}
		0 \to \ell^2(G) \longrightarrow  \bigoplus_{X}\ell^2(G) \longrightarrow \ell^2(G)e \longrightarrow 0.
	\end{equation*}
    
    Next, applying Lemma 4.5 in \cite{Pooya-Ren-Wang} to the above cochain complex, we obtain the following direct-sum isomorphism of vector spaces:
    $$
    \ell^2(G) \oplus (\ell^2(G)e) \oplus \ker\Delta_1 \cong \bigoplus_{X}\ell^2(G).
    $$
    Recall that the kernels of the Laplace operators in degrees other than 1 all vanish by Lemma \ref{lemma: l2 Betti number}.
    The proof of Proposition 4.6 in \cite{Pooya-Ren-Wang} shows that this isomorphism and its inverse are implemented by matrices over $\Csr(G)$. Passing to the associated projections therefore gives
    \[
        [1]+[e]+[p_1]=d[1]
        \quad\text{in }\mathrm K_0(\Csr(G)),
    \]
    which yields the stated formula.
    This finishes the proof.
\end{proof}

\begin{corollary}
\label{corollary: l2 Betti numbers for one relator groups}
    Under the hypotheses of Theorem \ref{theorem: one relator groups}, if $G$ is also hyperbolic, then its delocalised $\ell^2$-Betti numbers are given by
    \begin{equation*}
        \beta^{(2)}_{1,\langle g\rangle}(G)
        =(d-1)\mathbf 1_{\{g=1_G\}}-\tau_{\langle g\rangle}(e).
    \end{equation*}
    If $m<\infty$, this becomes
    \begin{equation*}
        \beta^{(2)}_{1,\langle g\rangle}(G)
        =(d-1)\mathbf 1_{\{g=1_G\}}
        -\frac{|\langle g\rangle\cap C|}{m},
    \end{equation*}
    where $C=\langle c\rangle$ is the finite cyclic subgroup determined by the relator. If $m=\infty$, then $e=0$, so the value is $d-1$ for $g=1_G$ and zero otherwise. In all cases, $\beta^{(2)}_{k,\langle g\rangle}(G)=0$ for $k\neq 1$.
\end{corollary}
\begin{proof}
    The conclusion follows immediately by applying the delocalised traces to the $K$-class expression of $[p_1]$ in Theorem \ref{theorem: one relator groups}.
    The hyperbolicity assumption guarantees the extension of the delocalised traces.
\end{proof}

In the following, we present several examples of one-relator groups and discuss their first higher Kazhdan projections and delocalised $\ell^2$-Betti numbers when the corresponding first spectral gap exists.

\begin{example}
Let $G$ be the fundamental group of a closed orientable surface $\Sigma$ of genus $g$:
$$
G=\pi_1(\Sigma)=\langle x_1,x_2,\ldots,x_{2g-1},x_{2g}\mid [x_1,x_2]\cdots [x_{2g-1},x_{2g}]\rangle,
$$
where $[x_{2i-1},x_{2i}]$ denotes the commutator $x_{2i-1}x_{2i}{x_{2i-1}}^{-1}{x_{2i}}^{-1}$.
We require that $g\geq 2$ to guarantee that $G$ is non-amenable. Since $\chi(G)=2-2g$ is non-zero in this case, we have $\alpha_1(\widetilde{\Sigma})={\infty}^+$.
The surface relator is not a proper power, so $m=1$ in this case.
Then, by Example 2.70 in \cite{LueckL2inv}, we obtain $\alpha_2(\widetilde{\Sigma})={\infty}^+$, which implies the existence of the first spectral gap by Proposition \ref{prop: spectral gap and NS inv}.
Since $G$ is torsion-free, Theorem \ref{theorem: one relator groups} shows that the $K$-class $[p_1]$ is a multiple of the identity:
\begin{equation*}
    [p_1]=(2g-2)[1]\in \mathrm K_0(\Csr(G)).
\end{equation*}
Moreover, since the surface group $G$ is hyperbolic when $g \geq 2$ (see Section 3.3 in \cite{Gromov}), we can use Corollary \ref{corollary: l2 Betti numbers for one relator groups} to compute the delocalised $\ell^2$-Betti numbers for $G$ as follows:
\begin{equation*}
    \beta^{(2)}_{1,\langle e\rangle}(G)=2g-2, \qquad
    \beta^{(2)}_{k,\langle s\rangle}(G)=0
    \quad\text{for all }(k,s)\neq(1,e).
\end{equation*}
Here $e$ denotes the group identity in $G$.
\end{example}

\begin{example}
    Consider the one-relator group $G=\bZ_n*\bZ=\langle a,b\mid a^n\rangle$ for some integer $n>1$.
    First notice that $\chi(G)=\frac{1}{n}-1$ is non-vanishing, hence $G$ is non-amenable, which implies that the first group cohomology $\mathrm H^1(G, \ell^2(G))$ is reduced by Lemma \ref{lemma: the first group cohomology}.
    We can verify the reducedness of the second group cohomology $\mathrm H^2(G, \ell^2(G))$ by direct computation and deduce that the first spectral gap exists by Proposition \ref{prop: spectral gap tool}.
    To prove that $\mathrm H^2(G, \ell^2(G))$ is reduced, it suffices to show that the image of the first codifferential with coefficients in $\ell^2(G)$ is closed.
    More precisely, by calculating the Fox derivative, the first codifferential is given by
    \begin{equation*}
        d^1: {\ell^2(G)}^{\oplus 2} \to \ell^2(G), \qquad
        d^1(f_1,f_2)=(1+a+\cdots+a^{n-1})f_1.
    \end{equation*}
    Define the averaging projection $e=\frac{1+a+\cdots+a^{n-1}}{n}$. Then $\operatorname{im}d^1=e\ell^2(G)$, which is closed because $e$ is a projection.
    Given the existence of the first spectral gap established above, Theorem \ref{theorem: one relator groups} gives
    $$
    [p_1]=[1]-[\frac{1+a+\cdots+a^{n-1}}{n}]\in \mathrm K_0(\Csr(G)).
    $$
    Moreover, since $\bZ_n$ and $\bZ$ are both hyperbolic groups, $G=\bZ_n*\bZ$ is also hyperbolic (see Section 3.2 of \cite{Gromov}).
    Applying Corollary \ref{corollary: l2 Betti numbers for one relator groups}, we obtain the delocalised $\ell^2$-Betti numbers
     \begin{equation*}
        \beta^{(2)}_{1,\langle g\rangle}(G)=
        \begin{cases}
            1-\frac{1}{n} & \text{if } g=1_G,\\
            -\frac{1}{n} & \text{if } g \text{ is conjugate to }a^j
                \text{ for some }1\leq j<n,\\
            0 & \text{otherwise},
        \end{cases}
    \end{equation*}
    and $\beta^{(2)}_{k,\langle g\rangle}(G)=0$ for $k\neq 1$ and $g\in G$.
\end{example}

We now provide an example of a non-amenable one-relator group for which the first spectral gap does not exist.
In this example, we use the Novikov--Shubin invariants to discuss the absence of the first spectral gap.
\begin{example}
    Consider the one-relator group $G=\bZ^2*\bZ=\langle a,b,c\mid [a,b]\rangle$.
    Then $EG$ admits a CW-complex model of finite type given by $\widetilde{T^2 \vee S^1}$.
    First notice that $\chi(G)=-1$ is non-zero, so $G$ is non-amenable, which implies $\alpha_1(\widetilde{T^2 \vee S^1})={\infty}^+$.
    For the second Novikov--Shubin invariant, the proof of Proposition 3.7 in \cite{Lott-Lueck} shows that $\alpha_2(\widetilde{T^2 \vee S^1})$ is the minimum of $\alpha_2(\widetilde{T^2})$ and $\alpha_2(\widetilde{S^1})$.
    Finally, from Example 2.59 in \cite{LueckL2inv}, one has 
    \begin{equation*}
        \alpha_2(\widetilde{T^2})=2,\qquad \alpha_2(\widetilde{S^1})={\infty}^+.
    \end{equation*}
    Therefore, $\alpha_2(\widetilde{T^2 \vee S^1})=2$. Hence, Proposition \ref{prop: spectral gap and NS inv} implies that the first spectral gap for $G$ does not exist, and neither does the first higher Kazhdan projection.
\end{example}

Finally, we end this paper with the following example of one-relator groups with exactly two generators, the Baumslag--Solitar groups.
\begin{example}
Consider the Baumslag--Solitar group $G=\mathrm{BS}(m,n)=\langle a,b\mid ba^m b^{-1}=a^n\rangle$ for nonzero integers $m,n$.
As the relation in this case is not a proper power, its exponent, in the sense defined above, is $1$.
Hence $\chi(G)$ vanishes, as do the $\ell^2$-Betti numbers in all degrees by Lemma \ref{lemma: l2 Betti number}.
By the faithfulness of the von Neumann dimension, all higher Kazhdan projections vanish in a matrix algebra over the von Neumann algebra $\mathcal{N}(G)$, and we do not consider the existence of spectral gaps when the corresponding kernel projection vanishes.
\end{example}

\bibliographystyle{amsplain}
\bibliography{main}

\end{document}